\documentclass{amsart}

\usepackage{graphicx}

\newtheorem{theorem}{Theorem}[section]
\newtheorem{lemma}[theorem]{Lemma}
\newtheorem{proposition}[theorem]{Proposition}

\theoremstyle{definition}
\newtheorem{definition}[theorem]{Definition}

\newtheorem{conjecture}{Conjecture}
\theoremstyle{remark}
\newtheorem{remark}[theorem]{Remark}

\numberwithin{equation}{section}

\begin{document}

\title[Minimal hypersurfaces with index one]{Compact minimal hypersurfaces of index one and the width of real projective spaces}

\author{Alejandra Ram\'irez-Luna}
\address{Departmento de Matem\'aticas, Universidad del Valle,‎ Calle 13 No. 100-00, Cali, Colombia.}
\email{maria.ramirez.luna@correounivalle.edu.co}
\thanks{The first author was supported by IMU and TWAS}

\date{July 29, 2019}

\dedicatory{}

\keywords{minimal hypersurface, projective space, Morse index}

\begin{abstract}
We characterize the first min-max width of real projective spaces of any dimension. The width is the minimum area over the
Clifford hypersurfaces. We also compute the Morse index of the Clifford hypersurfaces in the complex and quaternionic projective spaces. 

\end{abstract}

\maketitle

\section{Introduction}

When studying an $n + 1$ dimensional Riemannian manifold $N$, minimal hypsersurfaces arise naturally. The existence of minimal hypersurfaces was treated by the works of Almgren \cite{almgren1965theory}, Pitts \cite{pitts2014existence}, Schoen and Simon \cite{schoen1981regularity}. They proved the existence of a closed minimal hypersurface $\Sigma$ whose singular set is of Hausdorff dimension no larger than $n-7$. This minimal hypersurface is called a min-max hypersurface and has some specific properties. \\

Let $M$ be a minimal hypersurface of $N$. The Morse index of $M$ is the number of negative eigenvalues of the Jacobi operator $\displaystyle J_{M}$ counting multiplicity. We say $M$ is stable if its Morse index is zero. This notion is important since it  gives us information about how unstable the hypersurface is. Marques and Neves \cite{marques2012rigidity} prove that any metric on a three-sphere which has scalar curvature greater than or equal to $6$ and is not round must have an embedded minimal sphere (the min-max hypersurface) of area strictly smaller than $4\pi$ and Morse index at most one. When the Ricci curvature is positive  Zhou \cite{zhou2017min} shows that the min-max hypersurface is either orientable and has Morse index one, or is a double cover of a non-orientable stable minimal hypersurface. The second option given by  Zhou was ruled out by Ketover, Marques and Neves \cite{ketover2016catenoid} when $4\leq (n+1)\leq 7$ and by the author \cite{maleja1} in general dimensions.\\

The study of minimal hypersurfaces according to their index in a given Riemannian manifold has been of great interest in geometry.  Since the main work of Simons \cite{simons1968minimal} several works have been developed around this topic. For example, Ohnita \cite{ohnita1986stable} studied these hypersurfaces when they are stable (Morse index equal to zero) in compact symmetric spaces of rank one, Torralbo and Urbano \cite{torralbo2014stable} studied these hypersurfaces when they are stable in the product of a sphere and any Riemannian Manifold and Perdomo \cite{perdomo2004average} \cite{perdomo2001low} has worked in low index minimal hypersurfaces of spheres.\\

This continues to be an active area of study. For example, Perdomo \cite{2019arXiv190210801P} has proved that if the index of a compact minimal hypersurface $M$ in $S^{n+1}$ is $n+3$ then the average $\int_{M}|\sigma^{2}|$ of the second fundamental form $\sigma$ cannot be very small. On the other hand, it has been shown that there exist relations between topological, geometrical and analytical properties. In particular it has been seen that in certain cases the Morse index of a minimal hypersurface controls its geometry and topology. For instance, Savo \cite{savo2010index} has shown that the index of a compact, orientable minimal hypersurface of the unit sphere is bounded below by a linear function of the first Betti number  (see \cite{ambrozio2018comparing} for a general Riemannian manifold) and Chodosh and Maximo \cite{chodosh2016topology} have proved that for an immersed two-sided minimal surface in $\mathbb{R}^{3}$, there is a lower bound of the index depending on the genus and number of ends.\\

The classification of minimal hypersurfaces in Riemannian manifolds according to their Morse index has played an important role in the development of geometry. In the sphere $S^{n+1}$, we define the Clifford hypersurfaces by the set $S^{n_{1}}(R_{1})\times S^{n_{2}}(R_{2})$ where $n_{1}$ and $n_{2}$ are positive integers such that $n_{1}+n_{2}=n$ and $R_{1}$ and $R_{2}$ are positive real numbers such that $R_{1}^{2}+R_{2}^{2}=1$. We also call Clifford hypersurface to the projection of the Clifford hypersurfaces in $S^{n+1}$ under the usual immersion into the real projective space. It has been conjectured that the Clifford hypersurfaces are the only minimal hypersurfaces  with index $n+3$ in $S^{n+1}$. This has been solved by Urbano \cite{urbano1990minimal} when $n=2$, in fact, he stated that the index of a compact orientable nontotally geodesic minimal surface $\displaystyle \Sigma$ in $\displaystyle S^{3}$ is greater than or equal to $5$, and the equality holds if and only if $\displaystyle \Sigma$ is the Clifford torus. This last theorem was fundamental in the celebrated proof of the  Willmore conjecture by Marques and Neves  \cite{marques2014min}.\\

In this paper we are going to use the classification of index one minimal hypersurfaces of real projective spaces by do Carmo, Ritor\'e and Ros in order to compute their first min-max widths. This holds in any dimension and uses work of the author \cite{maleja1}. Viana also computed the width of the real projective space \cite{celso} and Batista and Lima did it in dimensions less than or equal to $7$ \cite{batista, batista2}.

\begin{theorem} \cite{do2000compact}
The only compact two-sided minimal hypersurfaces with index one in the real projective space $\displaystyle \mathbb{R}P^{r-1}$ are the totally geodesic spheres and the minimal Clifford hypersurfaces. 
\label{maintheoremreal}
\end{theorem}

\begin{remark}
In \cite{do2000compact}, do Carmo, Ritor\'e and Ros prove the smooth version
and point out that the extension to the singular case follows as in \cite{morgan2002isoperimetric}.
\end{remark}

Consequently, we used Theorem \ref{maintheoremreal} to classify minimal hypersurfaces obtained by the min-max technique applied to one-parameter sweepouts (see section 4 for more details).

\begin{theorem} Let $\displaystyle \Sigma_{i}$ be the min-max hypersurface in $\displaystyle \mathbb{R}P^{i}$; $i\in \mathbb{N}$. Then $\displaystyle \Sigma_{i}$ is the minimal Clifford hypersurface of least area. In particular,

\begin{center}
$\displaystyle W(\mathbb{R}P^{i})=|\Sigma_{i}|=\left\{ \begin{array}{lcc}
              |\Pi_{\mathbb{R}}(S^{1}_{\sqrt{\frac{1}{2}}}\times S^{1}_{\sqrt{\frac{1}{2}}})|=\pi^{2} &  {\rm if}  & i=3 \\
              |\Pi_{\mathbb{R}}(S^{1}_{\sqrt{\frac{1}{3}}}\times S^{2}_{\sqrt{\frac{2}{3}}})|=\frac{8\pi^{2}}{3\sqrt{3}} &  {\rm if}  & i=4\\
              |\Pi_{\mathbb{R}}(S^{2}_{\sqrt{\frac{2}{4}}}\times S^{2}_{\sqrt{\frac{2}{4}}})|=2\pi^{2} &  {\rm if} & i=5\\
               |\Pi_{\mathbb{R}}(S^{2}_{\sqrt{\frac{2}{5}}}\times S^{3}_{\sqrt{\frac{3}{5}}})|= \frac{24}{25}\sqrt{\frac{3}{5}}\pi^{3} &  {\rm if} & i=6\\
              |\Pi_{\mathbb{R}}(S^{3}_{\sqrt{\frac{3}{6}}}\times S^{3}_{\sqrt{\frac{3}{6}}})|=\frac{\pi^{4}}{4} &  {\rm if} & i=7
             \end{array}
   \right.$
\end{center}
where $W$ is the width.
\label{cori1}
\end{theorem}

We also find an upper bound for the area of the min-max hypersurface in $\mathbb{C}P^{i}$, $i\in \{2,3\}$

\begin{theorem} Let $\displaystyle \Sigma_{i}$ be  the min-max hypersurface in $\displaystyle \mathbb{C}P^{i}$, $\displaystyle i\in\{2,3\}$. Then $\displaystyle \Sigma_{i}$ is a minimal Clifford hypersurface and  
\begin{center}
$\displaystyle W(\mathbb{C}P^{i})=|\Sigma_{i}|\leq \left\{ \begin{array}{lcc}
             |\Pi_{\mathbb{C}}(S^{1}_{\sqrt{\frac{1}{4}}}\times S^{3}_{\sqrt{\frac{3}{4}}})|=\frac{3\sqrt{3}\pi^{2}}{8} &  {\rm if}  & i=2 \\
              |\Pi_{\mathbb{C}}(S^{3}_{\sqrt{\frac{3}{6}}}\times S^{3}_{\sqrt{\frac{3}{6}}})|=\frac{\pi^{3}}{4} &  {\rm if} & i=3
             \end{array}
   \right.$
\end{center}
where $W$ is the width.
\label{cori2}
\end{theorem}

Additionally, the theory developed in the proof of the Theorem \ref{maintheoremreal} inspired us to find the index of minimal Clifford Hypersurfaces in the complex and quaternionic space 
\begin{theorem}
The index of the Clifford hypersurfaces in $\displaystyle \mathbb{K}P^{r-1}$ is one, $\mathbb{K}\in \{\mathbb{R}, \mathbb{C}\}$
\label{modestia}
\end{theorem}
and to conjecture a generalization of Theorem \ref{maintheoremreal}

\begin{conjecture}
The only compact minimal hypersurfaces with index one in the projective space $\displaystyle \mathbb{K}P^{r-1}$ are the minimal Clifford hypersurfaces. In particular they can be minimal geodesic spheres. 
\label{maintheoremcomplex}
\end{conjecture}

We point out that if the Conjecture \ref{maintheoremcomplex} is true, then we have equality in Theorem \ref{cori2}.

We have organized this paper in three sections. The first section gives some basic definitions, notations and facts needed along the paper. In the second section we related a compact minimal hypersurface $\bar{M}$ in $\displaystyle \mathbb{K}P^{r-1}$ with its inverse image $\displaystyle M:=\Pi_{\mathbb{K}}^{-1}(\bar{M})$ in $S^{dr-1}$,
where $\mathbb{K}\in \{\mathbb{C}, \mathbb{H} \}$, $d=\dim_{\mathbb{R}}\mathbb{K}$ and $\displaystyle \Pi_{\mathbb{K}}$ is the usual submersion (see preliminary concepts section). We also stablished some features of $M$ which are inherited from $\displaystyle \bar{M}$ such as minimality, orientability and characteristics of the Jacobi operator (see Remark \ref{remarkfeatures} and Lemma \ref{features}). The third section is an application to the min-max hypersurface in $\displaystyle \mathbb{C}P^{r}$, $\displaystyle \mathbb{R}P^{3}$,$\displaystyle \mathbb{R}P^{4}$, $\displaystyle \mathbb{R}P^{5}$, $\displaystyle \mathbb{R}P^{6}$ and  $\displaystyle \mathbb{R}P^{7}$; $r\in \mathbb{N}$.  \\

\textbf{Acknowledgments}
This study was made possible thanks to a PhD scholarship (IMU Breakout Graduate Fellowship) from IMU and TWAS to the author. I  am very grateful for the excellent direction of Professor Fernando Cod\'a Marques. I deeply value the meaningful conversations with Professor Gonzalo Garc\'ia.  Part of this work was done while the author was visiting Princeton University as a VSRC. I am grateful to Princeton University for the hospitality. I am also thankful to the Deparment of Mathematics at Universidad del Valle for partial support my visit to Princeton.

\section{Preliminary concepts}
In this first section we will give some important definitions and basic properties given in \cite{barbosa2012stability} and
\cite{do2000compact}.  Let $\displaystyle S^{N}_{R}$ be the Euclidean $N$- sphere of radius $R$, $d=\dim_{\mathbb{R}}(\mathbb{K})$ where $\mathbb{K}\in \{\mathbb{C}, \mathbb{H}\}$. The sphere $\displaystyle S^{d-1}$  acts freely on $\displaystyle S^{dr-1}_{1}\subset \mathbb{R}^{dr}=\mathbb{K}^{r}$ by left scalar multiplication. Then we define $\displaystyle \mathbb{K}P^{r-1}:= S^{dr-1}_{1}/ S^{d-1}$, i.e. the orbit space. The manifold $\displaystyle \mathbb{K}P^{r-1}$ is called  \textbf{the complex projective space} (resp. \textbf{the quaternionic projective space}) of real dimension $\displaystyle 2r-2$ (resp. $4r-4$) when $\mathbb{K}=\mathbb{C}$ (resp. $\mathbb{K}=\mathbb{H}$). Action of $S^{d-1}$ induces the usual Riemannian submersion
\begin{center}
$\displaystyle \Pi_{\mathbb{K}}:S_{1}^{dr-1}\rightarrow \mathbb{K}P^{r-1} $,
\end{center}
where $\displaystyle \Pi_{\mathbb{K}}(p)$ is the orbit of $p$.

\begin{definition}
A \textbf{Clifford hypersurface} is the set of points $\displaystyle S^{n_{1}}_{R_{1}} \times S^{n_{2}}_{R_{2}}$, such that $\displaystyle n_{i}\in \mathbb{Z}^{+}$ and $\displaystyle R_{i} \in \mathbb{R}^{+}$, where $\displaystyle n_{1}+n_{2}=dr-2$  and  $\displaystyle R_{1}^{2}+R_{2}^{2}=1$. When 
\begin{center}
$n_{1}\equiv -1 \mod d$ 
\end{center}
 we will call also as Clifford hypersurface the set $\displaystyle \Pi_{\mathbb{K}}(S^{n_{1}}_{R_{1}} \times S^{n_{2}}_{R_{2}})$ in $\mathbb{K}P^{r-1}$. 
\end{definition}

\begin{remark}
The manifold $\displaystyle S^{n_{1}}_{R_{1}} \times S^{n_{2}}_{R_{2}}$ is minimal if and only if $\displaystyle n_{1}R_{2}^{2}=n_{2}R_{1}^{2}$.
\end{remark}

\begin{remark} \cite{barbosa2012stability}
Let $\displaystyle U_{\rho}(.)$ be the tubular neighborhood of radius $\displaystyle \rho$, $\displaystyle c=\cos(\rho)$ and $\displaystyle s=\sin(\rho)$. 
Since
\begin{center}
$\displaystyle S_{c}^{d-1}\times S_{s}^{d(r-1)-1}=\partial U_{\rho}(S^{d-1})=\partial U_{\rho}(\Pi_{\mathbb{K}}^{-1}(\mathbb{K}P^{0}))=
\Pi_{\mathbb{K}}^{-1}(\partial U_{\rho}(\mathbb{K}P^{0}))$,\\
\end{center}
then the hypersurface $\Pi_{\mathbb{K}}(\displaystyle S_{c}^{d-1}\times S_{s}^{d(r-1)-1}) $ corresponds to a geodesic sphere in $\displaystyle \mathbb{K}P^{r-1}$. 
\end{remark}
 
Let $\displaystyle f:M\rightarrow S^{dr-1}$ be a minimal compact orientable hypersurface, $\displaystyle N$ its normal vector field and $\displaystyle |\sigma|$ the norm of the second fundamental form of $f$. Then we have the equations
\begin{equation}
\bigtriangleup_{M} f + (dr-2)f=0 \,\,\,\,\,\,\,\, {\rm and} \,\,\,\,\,\,\,\, \bigtriangleup_{M} N + |\sigma |^{2}N=0.
\label{equationsfn}
\end{equation}  
 
 \begin{definition}
 Let $\displaystyle J_{M}$ be the \textbf{Jacobi operator} of $M$ given  by $\displaystyle J_{M}=\bigtriangleup_{M}+|\sigma |^{2}+dr-2$ and its quadratic form as
 
 \begin{center}
 $\displaystyle \displaystyle Q(u,u)=-\int_{M}uJ_{M}(u)dV$,
 \end{center}
 where $\displaystyle u:M\rightarrow \mathbb{R}$.
 
 \begin{definition}
 The \textbf{Morse index} of $\displaystyle M$ is the number of negative eigenvalues of $\displaystyle J_{M}$ counting multiplicities. 
 \end{definition}

 \end{definition}

\section{Results}
\begin{proposition}
 Let $\displaystyle \bar{f}:\bar{M}\rightarrow  \mathbb{K}P^{r-1}$ be a compact oriented minimal hypersurface in $\displaystyle \mathbb{K}P^{r-1}$. Denote by 
  $\displaystyle f:M^{dr-2}\rightarrow S^{dr-1}\subset \mathbb{R}^{dr}=\mathbb{K}^{r}$ the uniquely defined compact hypersurface such that  $\displaystyle \Pi_{\mathbb{K}}(f(M))=\bar{f}(\bar{M})$. Then $\displaystyle f$ is oriented and minimal. 
  \label{propositionminimality}
\end{proposition}

\begin{proof}
Let $\displaystyle \bar{N}$ be a unitary normal vector field on $\displaystyle \bar{M}$. Consider the unique horizontal lift $\displaystyle N$ of $\displaystyle \bar{N}$ to $M$. Then $N$ is orthogonal to the fibers, $\displaystyle \Pi_{\mathbb{K}}$- related to $\displaystyle \bar{N}$ ($\displaystyle \Pi_{\mathbb{K}*}(N)=\bar{N}$) and normal to $M$ in $\displaystyle S^{dr-1}$.\\
Denote the second fundamental tensors by $\displaystyle B(v):=-(\nabla_{v}N)^{T}$ and $\displaystyle \bar{B}(w):=(\nabla_{w}\bar{N})^{T}$.  By O'Neills  Formulas \cite{o1966fundamental} we have for $\displaystyle u,v\in T_{p}M$ which are horizontal 
\begin{equation}
<B(u),v>=<\bar{B}(\Pi_{\mathbb{K}*}(u)), \Pi_{\mathbb{K}*}(v)>.
\end{equation}

Consider $\displaystyle \{\bar{e_{1}},\bar{e_{2}},...,\bar{e}_{d(r-1)-1}\}$ an orthonormal basis in $\displaystyle T_{\Pi_{\mathbb{K}}(p)}\bar{M}$. For each $\displaystyle \bar{e_{i}}$ consider its unique horizontal lift $\displaystyle e_{i}$, $\displaystyle i=1,...,d(r-1)-1$. Let $\displaystyle e_{d(r-1)},...,e_{dr-2}$ be the unit vectors tangent to the geodesic spheres $S^{1}$ which generate the orbit $S^{d-1}$ of   $\displaystyle \Pi_{\mathbb{K}}$. Therefore $\displaystyle \{e_{1},...,e_{dr-2}\}$ is an orthonormal basis in $\displaystyle T_{p}M$. If $\displaystyle H$ and $\displaystyle \bar{H}$ are the mean curvature of $\displaystyle M$ and $\displaystyle \bar{M}$ respectively, then

\begin{equation*}
\begin{split}
\displaystyle \displaystyle H=&\sum_{i=1}^{d(r-1)-1}<\nabla_{e_{i}}e_{i},N>+<\nabla_{e_{d(r-1)}}e_{d(r-1)},N>+...+<\nabla_{e_{dr-2}}e_{dr-2},N>\\ 
\displaystyle=&-\sum_{i=1}^{d(r-1)-1}<e_{i},\nabla_{e_{i}}N>+<\nabla_{e_{d(r-1)}}e_{d(r-1)},N>+...+<\nabla_{e_{dr-2}}e_{dr-2},N>\\
\displaystyle \displaystyle
=&\sum_{i=1}^{d(r-1)-1}<e_{i},B(e_{i})>+<\nabla_{e_{d(r-1)}}e_{d(r-1)},N>+...+<\nabla_{e_{dr-2}}e_{dr-2},N>\\
\displaystyle \displaystyle
=&\sum_{i=1}^{2r-3}<\bar{e}_{i},\bar{B}(\bar{e}_{i})>+<\nabla_{e_{d(r-1)}}e_{d(r-1)},N>+...+<\nabla_{e_{dr-2}}e_{dr-2},N>\\
\displaystyle =&\bar{H}+<\nabla_{e_{d(r-1)}}e_{d(r-1)},N>+...+<\nabla_{e_{dr-2}}e_{dr-2},N>.
\end{split}
\end{equation*} 
Since $\displaystyle e_{d(r-1)},...,e_{dr-2}$ are tangent vectors to a geodesic, then 
\begin{center}
$<\nabla_{e_{d(r-1)}}e_{d(r-1)},N>=...=<\nabla_{e_{dr-2}}e_{dr-2},N>=0$,
\end{center}
 Therefore $\displaystyle H=\bar{H}$.
\end{proof}

\begin{definition} Under the hypothesis of the last proposition, we define $\displaystyle u:M\rightarrow \mathbb{R}$ to be a \textbf{ $\displaystyle S^{d-1}$-equivariant function} if for $\displaystyle p$ and $\displaystyle q$ in the same orbit we have $\displaystyle u(p)=u(q)$. An eigenvalue is $\displaystyle S^{d-1}$-equivariant if its associated eigenfunction is $\displaystyle S^{d-1}$-equivariant. 

\end{definition}

\begin{remark}
The features of $\displaystyle J_{\bar{M}}$ in $\displaystyle \bar{M}$ give information about $\displaystyle J_{M}$. If $\displaystyle \bar{u}$ is a smooth function on $\displaystyle \bar{M}$, let $\displaystyle u$ be the $\displaystyle S^{d-1}$-equivariant function on $\displaystyle M$ given by $\displaystyle u=\bar{u}\circ \Pi_{\mathbb{K}}$. By the remark in section 4 in \cite{barbosa2012stability} $\displaystyle J_{M}(u)=(J_{\bar{M}}\bar{u})\circ \Pi_{\mathbb{K}}$. It is clear that if $\displaystyle u$ is a $\displaystyle S^{d-1}$-equivariant function on $\displaystyle M$, we can define a function $\displaystyle \bar{u}:\bar{M} \rightarrow \mathbb{R}$ such that  $\displaystyle u=\bar{u}\circ \Pi_{\mathbb{K}}$. In particular 

\begin{itemize}
\item set of eigenvalues of $\displaystyle J_{\bar{M}}\subset$ set of eigenvalues of $\displaystyle J_{M}$.

\item  
set of eigenvalues of $\displaystyle J_{\bar{M}}$= set of $\displaystyle S^{d-1}$-equivariant eigenvalues of $\displaystyle J_{M}$.
\end{itemize}
\label{remarkfeatures}
\end{remark}

\begin{lemma} Under the assumptions of the last proposition let $\bar{M}$ be a hypersurface with index one, $\displaystyle u:M\rightarrow \mathbb{R}$ be a smooth $\displaystyle S^{d-1}$- equivariant function and $\displaystyle \varphi$ the first eigenfunction of $\displaystyle J_{M}$ then

\begin{enumerate}

\item  $\displaystyle \varphi$ is $\displaystyle S^{d-1}$-equivariant.

\item Let $\displaystyle u$ be a  $\displaystyle S^{d-1}$- equivariant map, such that $\displaystyle \int_{M}\varphi udV=0$ then 
\begin{center}
$\displaystyle Q(u,u)\geq 0$.
\end{center}

\end{enumerate}
\label{features}
\end{lemma}
 \begin{proof}

\begin{enumerate}

\item 

Let $\displaystyle x\in S^{d-1}$ be an arbitrary point in the sphere. Define the isometry $\displaystyle A_{x}:M\rightarrow M$  by $\displaystyle A_{x}(p)=x\cdot p$. The map $\displaystyle \varphi \circ A_{x}$ is another eigenfunction associated to the first eigenvalue $\displaystyle \lambda_{1}$ of $\displaystyle J_{M}$. The dimension of $\displaystyle \lambda_{1}$-space is one, then $\displaystyle k \varphi \circ A_{x} = \varphi$; where $\displaystyle k$ is a constant. Since $\displaystyle \varphi\geq 0$ then $\displaystyle k\geq 0$. On the other hand 
\begin{center}
$\displaystyle \displaystyle k^{2} \int_{M}( \varphi \circ A_{x})^{2}dV=\int_{M}\varphi^{2}dV$,
\end{center}
thus $\displaystyle k^{2}=1$. Therefore $\displaystyle k=1$ and $\displaystyle \varphi \circ A_{x} = \varphi$, i.e. $\displaystyle \varphi$ is $\displaystyle S^{d-1}$-equivariant.

\item
 Since $\displaystyle u$ is $\displaystyle S^{d-1}$-equivariant, there exists $\displaystyle \bar{u}:\bar{M}\rightarrow \mathbb{R}$ such that $\displaystyle u=\bar{u}\circ \Pi_{\mathbb{K}}$. Then 

\begin{center}

$\displaystyle \displaystyle |S^{d-1}|\int_{\bar{M}} \bar{\varphi}\bar{u}
d\bar{V}=\int_{M}\varphi udV=0$
\end{center}
and
\begin{center}
$\displaystyle \displaystyle 0\leq Q(\bar{u},\bar{u})=-\int_{\bar{M}}\bar{u}J_{\bar{M}}(\bar{u})
d\bar{V}=-\frac{1}{|S^{d-1}|}\int_{M}uJ_{M}(u)
dV$\\
$\displaystyle =\frac{1}{|S^{d-1}|}Q(u,u)$.
\end{center} 
where $|S^{d-1}|$ is the area of $S^{d-1}$.

\end{enumerate}

 \end{proof}

\begin{theorem}
The index of the Clifford hypersurfaces in $\displaystyle \mathbb{K}P^{r-1}$ is one. 
\end{theorem}
\begin{proof}
Let $\displaystyle \Pi_{\mathbb{K}}(S^{n_{1}}_{R_{1}}\times S^{n_{2}}_{R_{2}})$ be a typical minimal Clifford hypersurface, where $\displaystyle n_{1}+n_{2}=dr-2$, $\displaystyle R_{1}^{2}+R_{2}^{2}=1$, $\displaystyle n_{1}R_{2}^{2}=n_{2}R_{1}^{2}$ and $n_{1}\equiv -1 \mod d$. We look for $\displaystyle S^{d-1}$-equivariant negative eigenvalues of $\displaystyle J_{M}$,  where $\displaystyle M=S^{n_{1}}_{R_{1}}\times S^{n_{2}}_{R_{2}}$ (see Remark \ref{remarkfeatures}). In the manifold $\displaystyle M$ we have

\begin{center}
$\displaystyle \displaystyle |\sigma|^{2}=n_{1}\frac{R_{2}^{2}}{R_{1}^{2}}+n_{2}\frac{R_{1}^{2}}{R_{2}^{2}}=\frac{n_{2}R_{1}^{2}}{R_{1}^{2}}+\frac{n_{1}R_{2}^{2}}{R_{2}^{2}}=n_{2}+n_{1}=dr-2.$
\end{center}

Since $\displaystyle J_{M}=\bigtriangleup_{M}+2(dr-2)$, we look for $\displaystyle S^{d-1}$-equivariant eigenvalues $\displaystyle \beta$ of $\displaystyle \bigtriangleup_{M}$, such that $\displaystyle \beta<2(dr-2)$. The eigenvalues of  $\displaystyle \bigtriangleup_{M}$ are given by

\begin{center}
$\displaystyle \displaystyle \beta_{k_{1}k_{2}}=\frac{k_{1}(k_{1}+n_{1}-1)}{R_{1}^{2}}+\frac{k_{2}(k_{2}+n_{2}-1)}{R_{2}^{2}}$
\end{center} 
where $\displaystyle k_{1},k_{2}\in  \mathbb{N}^{0}$.

Notice that $\displaystyle \beta_{00}=0$ corresponds to the constant functions which are $\displaystyle S^{d-1}$-equivariant. When $\displaystyle k_{1}=1=k_{2}$, 

\begin{center}
$\displaystyle \beta_{11}=\frac{n_{1}}{R_{1}^{2}}+\frac{n_{2}}{R_{2}^{2}}=n_{1}(1+\frac{R_{2}^{2}}{R_{1}^{2}})+n_{2}(1+\frac{R_{1}^{2}}{R_{2}^{2}})=n_{1}\frac{R_{2}^{2}}{R_{1}^{2}}+n_{2}\frac{R_{1}^{2}}{R_{2}^{2}}+n_{1}+n_{2}=2(dr-2)$.
\end{center}
Therefore we get rid of $\displaystyle \beta_{k_{1}k_{2}}$, when $\displaystyle k_{1}\geq 1$ and $\displaystyle k_{2}\geq 1$. Using the same argument given in \cite{barbosa2012stability} the eigenfunctions corresponding to $\displaystyle \beta_{10}$(respectively $\displaystyle \beta_{01}$) are linear functions on $\displaystyle \mathbb{R}^{n_{1}+1}$(respectively $\displaystyle \mathbb{R}^{n_{2}+1}$) which are never $\displaystyle S^{d-1}$-equivariant since $\displaystyle -id\in S^{d-1}$. 
It suffices to show that $\displaystyle \beta_{20}$ and $\displaystyle \beta_{02}$  are greater than $\displaystyle \beta_{11}$. Since $\displaystyle n_{1}R_{2}^{2}=n_{2}R_{1}^{2}$ 
\begin{equation}
\frac{R_{2}^{2}}{R_{1}^{2}}=\frac{n_{2}}{n_{1}}\geq \frac{n_{2}}{n_{1}+2}
\label{des1}
\end{equation}
\begin{equation}
\frac{R_{2}^{2}}{R_{1}^{2}}=\frac{n_{2}}{n_{1}}\leq \frac{n_{2}+2}{n_{1}}.
\label{des2}
\end{equation}
But equations (\ref{des1}) and (\ref{des2}) are equivalent to

\begin{center}
$\displaystyle \displaystyle \beta_{20}=\frac{2(n_{1}+1)}{R_{1}^{2}}\geq \frac{n_{1}}{R_{1}^{2}}+\frac{n_{2}}{R_{2}^{2}}=\beta_{11}$
\end{center}

\begin{center}
$\displaystyle \displaystyle \beta_{02}=\frac{2(n_{2}+1)}{R_{2}^{2}}\geq \frac{n_{1}}{R_{1}^{2}}+\frac{n_{2}}{R_{2}^{2}}=\beta_{11}$.
\end{center}

\end{proof}

\section{Application}
Let $\displaystyle (M^{n+1},g)$ be a Riemannian manifold connected closed orientable and $\displaystyle H^{n}$ the $\displaystyle n$ dimensional Hausdorff measure. When $\displaystyle \Sigma^{n}$ is a submanifold, we use $\displaystyle |\Sigma|$ to denote $\displaystyle H^{n}(\Sigma)$. Let $Z_{n}(M^{n+1})$ be the space of integral cycles and $\Phi:[0,1]\rightarrow Z_{n}(M^{n+1})$ (which will be denoted by $\{\Phi_{s}\}_{s=0}^{1}$) be a sweepout. The reader can find in Section 4 of \cite{zhou2017min} the precise definition of sweepouts in Almgren-Pitts min-max theory. Denote by $\Pi$ all the mappings $\Lambda$ which are homotopic to $\Phi$ in $Z_{n}(M^{n+1})$. We define the width of $M^{n+1}$ as\\

\begin{center}
$\displaystyle W(M):=\inf_{\Lambda \in \Pi}\{\max_{x\in[0,1]}\mathcal{H}^{n}(\Lambda(x))\}$.
\end{center}

\begin{theorem}\cite{ketover2016catenoid, marques2012rigidity, maleja1,  zhou2015min}
Let $\displaystyle M^{n+1}$ be connected closed orientable Riemannian manifold with positive Ricci curvature. Then 
\begin{itemize}
\item the min-max minimal hypersurface $\displaystyle \Sigma$ is orientable of multiplicity one, which has Morse index one and $\displaystyle |\Sigma|=W(M)$.
\item $\displaystyle \displaystyle W(M)=\min_{\sigma \in S} \Big\{ |\sigma|\,\,\, {\rm if}\,\,\, \sigma \,\,\, {\rm is \,\,\, orientable\,\,\, or}\,\,\, 2A(\sigma)\,\,\, {\rm if}\,\,\, \sigma \,\,\, {\rm is \,\,\, non-orientable}\,\,\, \Big\}$,
\end{itemize}
where
\begin{center}
$\displaystyle S=\{\sigma^{n}\subset M^{n+1}: \sigma {\rm \,\,\,is\,\,\, a\,\,\, minimal\,\,\, hypersurface\,\,\, in}\,\,\, M\}$.
\end{center}
\label{xin}
\end{theorem}

\begin{remark}
\label{twoside}
The proof of the previous theorem also applies if we get rid of the assumption that $M$ is orientable, we consider the Almgren-Pitts width with $\mathbb{Z}_{2}$ coefficients and replace the orientability of  $\Sigma $  by two-sidedness.
\end{remark}

\begin{theorem} Let $\displaystyle \Sigma_{i}$ be  the min-max hypersurface in $\displaystyle \mathbb{C}P^{i}$, $\displaystyle i\in\{2,3\}$. Then   
\begin{center}
$\displaystyle \displaystyle W(\mathbb{C}P^{i})=|\Sigma_{i}|\leq\left\{ \begin{array}{lcc}
             |\Pi_{\mathbb{C}}(S^{1}_{\sqrt{\frac{1}{4}}}\times S^{3}_{\sqrt{\frac{3}{4}}})|=\frac{3\sqrt{3}\pi^{2}}{8} &  {\rm if}  & i=2 \\
              |\Pi_{\mathbb{C}}(S^{3}_{\sqrt{\frac{3}{6}}}\times S^{3}_{\sqrt{\frac{3}{6}}})|=\frac{\pi^{3}}{4} &  {\rm if} & i=3.
             \end{array}
   \right.$
\end{center}
\end{theorem}

\begin{proof} From Theorem \ref{xin} $\displaystyle \Sigma_{i}$  has Morse index one and 
\begin{center}
$\displaystyle W(\mathbb{C}P^{i})=|\Sigma_{i}|$.
\end{center}
The minimal Clifford hypersurfaces  in $\displaystyle \mathbb{C}P^{r-1}$ are $\displaystyle \Pi_{\mathbb{C}}(S^{n_{1}}_{R_{1}}\times S_{R_{2}}^{n_{2}})$ such that
\begin{center} $\displaystyle n_{1}+n_{2}=2r-2$, $\displaystyle R_{1}=\sqrt{\frac{n_{1}}{2r-2}}$, $\displaystyle R_{2}=\sqrt{\frac{n_{2}}{2r-2}}$, $n_{1}$, $n_{2}$ odd
\end{center}
 and
\begin{equation}
2\pi |\Pi_{\mathbb{C}}(S^{n_{1}}_{R_{1}}\times S_{R_{2}}^{n_{2}})|=|S^{n_{1}}_{R_{1}}\times S_{R_{2}}^{n_{2}}|=
\frac{4\pi^{\frac{n_{1}+n_{2}+2}{2}} R_{1}^{n_{1}} R_{2}^{n_{2}}}{\Gamma(\frac{n_{1}+1}{2})\Gamma(\frac{n_{2}+1}{2})}.
\label{areasclifford} 
\end{equation}

\begin{itemize} 
 \item In the case $\displaystyle \mathbb{C}P^{2}$ we only have one candidate to be the minimal Clifford hypersurface, 

\begin{center}
$\displaystyle \Pi_{\mathbb{C}}(S^{1}_{\sqrt{\frac{1}{4}}}\times S^{3}_{\sqrt{\frac{3}{4}}})$.
\end{center}
From equation (\ref{areasclifford})
\begin{center}
$\displaystyle \displaystyle |\Pi_{\mathbb{C}}(S^{1}_{\sqrt{\frac{1}{4}}}\times S^{3}_{\sqrt{\frac{3}{4}}})|=\frac{3\sqrt{3}\pi^{2}}{8},$
\end{center}
thus

\begin{center} $\displaystyle W(\mathbb{C}P^{2})\leq \frac{3\sqrt{3}\pi^{2}}{8}$.\end{center}

\item  In the case $\displaystyle \mathbb{C}P^{3}$ we have two candidate to be the minimal Clifford hypersurface,

\begin{center}
$\displaystyle \Pi_{\mathbb{C}}(S^{1}_{\sqrt{\frac{1}{6}}}\times S^{5}_{\sqrt{\frac{5}{6}}})$ and 
$\displaystyle \Pi_{\mathbb{C}}(S^{3}_{\sqrt{\frac{3}{6}}}\times S^{3}_{\sqrt{\frac{3}{6}}})$.
\end{center}
From equation (\ref{areasclifford})
\begin{center}
$\displaystyle \displaystyle |\Pi_{\mathbb{C}}(S^{1}_{\sqrt{\frac{1}{6}}}\times S^{5}_{\sqrt{\frac{5}{6}}})|=\frac{25\sqrt{5}\pi^{3}}{216}$
\end{center}

\begin{center}
$\displaystyle \displaystyle |\Pi_{\mathbb{C}}(S^{3}_{\sqrt{\frac{3}{6}}}\times S^{3}_{\sqrt{\frac{3}{6}}})|=\frac{\pi^{3}}{4}$
\end{center}
thus

\begin{center}
 $\displaystyle \displaystyle W(\mathbb{C}P^{3})\leq\min \{\frac{25\sqrt{5}\pi^{3}}{216}, \frac{\pi^{3}}{4} \}=\frac{\pi^{3}}{4}$.
\end{center}
\end{itemize}
\end{proof}
\begin{remark} In the last proof we got rid of $S^{4}_{1}$, $\displaystyle S^{2}_{\sqrt{\frac{2}{4}}}\times S^{2}_{\sqrt{\frac{2}{4}}}$, $S^{6}_{1}$ and $\displaystyle S^{4}_{\sqrt{\frac{4}{6}}} \times S^{2}_{\sqrt{\frac{2}{6}}}$ because they are not $\displaystyle S^{1}$-equivariant, even after a rotation.
\end{remark}

\begin{theorem} Let $\displaystyle \Sigma_{i}$ be the min-max hypersurface in $\displaystyle \mathbb{R}P^{i}$; $i\in \mathbb{N}$.  Then $\displaystyle \Sigma_{i}$ is the minimal Clifford hypersurface of least area. In particular, 

\begin{center}
$\displaystyle W(\mathbb{R}P^{i})=|\Sigma_{i}|=\left\{ \begin{array}{lcc}
              |\Pi_{\mathbb{R}}(S^{1}_{\sqrt{\frac{1}{2}}}\times S^{1}_{\sqrt{\frac{1}{2}}})|=\pi^{2} &  {\rm if}  & i=3 \\
              |\Pi_{\mathbb{R}}(S^{1}_{\sqrt{\frac{1}{3}}}\times S^{2}_{\sqrt{\frac{2}{3}}})|=\frac{8\pi^{2}}{3\sqrt{3}} &  {\rm if}  & i=4\\
              |\Pi_{\mathbb{R}}(S^{2}_{\sqrt{\frac{2}{4}}}\times S^{2}_{\sqrt{\frac{2}{4}}})|=2\pi^{2} &  {\rm if} & i=5\\
               |\Pi_{\mathbb{R}}(S^{2}_{\sqrt{\frac{2}{5}}}\times S^{3}_{\sqrt{\frac{3}{5}}})|= \frac{24}{25}\sqrt{\frac{3}{5}}\pi^{3} &  {\rm if} & i=6\\
              |\Pi_{\mathbb{R}}(S^{3}_{\sqrt{\frac{3}{6}}}\times S^{3}_{\sqrt{\frac{3}{6}}})|=\frac{\pi^{4}}{4} &  {\rm if} & i=7.
             \end{array}
   \right.$
\end{center}
\end{theorem}
\begin{proof}
By Theorem \ref{xin} if $i$ is odd and Remark \ref{twoside} if $i$ even, we have $\displaystyle \Sigma_{i}$ has index one and 
\begin{center}
$\displaystyle W(\mathbb{R}P^{i})=|\Sigma_{i}|.$
\end{center}
Therefore using Theorem $\displaystyle \ref{maintheoremreal}$ $\displaystyle \Sigma_{i}$ has to be a minimal Clifford hypersurface. The minimal Clifford hypersurfaces  in $\displaystyle \mathbb{R}P^{r-1}$ are $\displaystyle \Pi_{\mathbb{R}}(S^{n_{1}}_{R_{1}}\times S_{R_{2}}^{n_{2}})$ such that
\begin{center}
 $\displaystyle n_{1}+n_{2}=r-2$, $\displaystyle R_{1}=\sqrt{\frac{n_{1}}{r-2}}$, $\displaystyle R_{2}=\sqrt{\frac{n_{2}}{r-2}}$,
\end{center} 
and
\begin{equation}
2 |\Pi_{\mathbb{R}}(S^{n_{1}}_{R_{1}}\times S_{R_{2}}^{n_{2}})|=|S^{n_{1}}_{R_{1}}\times S_{R_{2}}^{n_{2}}|=
\frac{4\pi^{\frac{n_{1}+n_{2}+2}{2}} R_{1}^{n_{1}}R_{2}^{n_{2}}}{\Gamma(\frac{n_{1}+1}{2})\Gamma(\frac{n_{2}+1}{2})}.
\label{areascliffordreal} 
\end{equation}

\begin{itemize} 
 \item In the case $\displaystyle \mathbb{R}P^{3}$ we only have one candidate to be the minimal Clifford hypersurface, 

\begin{center}
$\displaystyle \Pi_{\mathbb{R}}(S^{1}_{\sqrt{\frac{1}{2}}}\times S^{1}_{\sqrt{\frac{1}{2}}}).$
\end{center}
From equation (\ref{areascliffordreal})
\begin{center}
$\displaystyle \displaystyle |\Pi_{\mathbb{R}}(S^{1}_{\sqrt{\frac{1}{2}}}\times S^{1}_{\sqrt{\frac{1}{2}}})|=\pi^{2}$
\end{center}
thus 
\begin{center}
$\displaystyle W(\mathbb{R}P^{3})= \pi^{2}.$
\end{center}

\item  In the case $\displaystyle \mathbb{R}P^{4}$ we only have one candidate to be the minimal Clifford hypersurface,

\begin{center}
$\displaystyle \displaystyle \Pi_{\mathbb{R}}(S^{1}_{\sqrt{\frac{1}{3}}}\times S^{2}_{\sqrt{\frac{2}{3}}})$
\end{center}
From equation (\ref{areascliffordreal})
\begin{center}
$\displaystyle \displaystyle |\Pi_{\mathbb{R}}(S^{1}_{\sqrt{\frac{1}{3}}}\times S^{2}_{\sqrt{\frac{2}{3}}})|=\frac{8\pi^{2}}{3\sqrt{3}}$
\end{center}
thus
\begin{center}
 $\displaystyle \displaystyle W(\mathbb{R}P^{4})=\frac{8\pi^{2}}{3\sqrt{3}}.$
\end{center}

\item  In the case $\displaystyle \mathbb{R}P^{5}$ we have two candidate to be the minimal Clifford hypersurface,

\begin{center}
$\displaystyle \Pi_{\mathbb{R}}(S^{1}_{\sqrt{\frac{1}{4}}}\times S^{3}_{\sqrt{\frac{3}{4}}})$ and 
$\displaystyle \Pi_{\mathbb{R}}(S^{2}_{\sqrt{\frac{2}{4}}}\times S^{2}_{\sqrt{\frac{2}{4}}})$
\end{center}
From equation (\ref{areascliffordreal})
\begin{center}
$\displaystyle \displaystyle |\Pi_{\mathbb{R}}(S^{1}_{\sqrt{\frac{1}{4}}}\times S^{3}_{\sqrt{\frac{3}{4}}})|=\frac{3\sqrt{3}\pi^{3}}{8}$
\end{center}

\begin{center}
$\displaystyle \displaystyle |\Pi_{\mathbb{R}}(S^{2}_{\sqrt{\frac{2}{4}}}\times S^{2}_{\sqrt{\frac{2}{4}}})|=2\pi^{2}$
\end{center}
thus

\begin{center}
 $\displaystyle \displaystyle W(\mathbb{R}P^{5})=\min \{2\pi^{2}, \frac{3\sqrt{3}\pi^{3}}{8}\}=2\pi^{2}.$
\end{center}


\item  In the case $\displaystyle \mathbb{R}P^{6}$ we have two candidate to be the minimal Clifford hypersurface,

\begin{center}
$\displaystyle \Pi_{\mathbb{R}}(S^{1}_{\sqrt{\frac{1}{5}}}\times S^{4}_{\sqrt{\frac{4}{5}}})$ and 
$\displaystyle \Pi_{\mathbb{R}}(S^{2}_{\sqrt{\frac{2}{5}}}\times S^{3}_{\sqrt{\frac{3}{5}}})$
\end{center}
From equation (\ref{areascliffordreal})
\begin{center}
$\displaystyle \displaystyle |\Pi_{\mathbb{R}}(S^{1}_{\sqrt{\frac{1}{5}}}\times S^{4}_{\sqrt{\frac{4}{5}}})|=\frac{128\pi^{3}}{75 \sqrt{5}}$
\end{center}

\begin{center}
$\displaystyle \displaystyle |\Pi_{\mathbb{R}}(S^{2}_{\sqrt{\frac{2}{5}}}\times S^{3}_{\sqrt{\frac{3}{5}}})|=\frac{24}{25}\sqrt{\frac{3}{5}}\pi^{3}$
\end{center}
thus

\begin{center}
 $\displaystyle \displaystyle W(\mathbb{R}P^{6})=\min \{\frac{128\pi^{3}}{75 \sqrt{5}}, \frac{24}{25}\sqrt{\frac{3}{5}}\pi^{3}\}=\frac{24}{25}\sqrt{\frac{3}{5}}\pi^{3}.$
\end{center}


\item  In the case $\displaystyle \mathbb{R}P^{7}$ we have three candidates to be the minimal Clifford hypersurface,

\begin{center}
$\displaystyle \Pi_{\mathbb{R}}(S^{1}_{\sqrt{\frac{1}{6}}}\times S^{5}_{\sqrt{\frac{5}{6}}})$ 
\end{center}
\begin{center} 
$\displaystyle \Pi_{\mathbb{R}}(S^{2}_{\sqrt{\frac{2}{6}}}\times S^{4}_{\sqrt{\frac{4}{6}}})$
\end{center}
\begin{center} 
$\displaystyle \Pi_{\mathbb{R}}(S^{3}_{\sqrt{\frac{3}{6}}}\times S^{3}_{\sqrt{\frac{3}{6}}})$
\end{center}
From equation (\ref{areascliffordreal})
\begin{center}
$\displaystyle \displaystyle |\Pi_{\mathbb{R}}(S^{1}_{\sqrt{\frac{1}{6}}}\times S^{5}_{\sqrt{\frac{5}{6}}})|=\frac{25\sqrt{5}\pi^{4}}{216}$
\end{center}

\begin{center}
$\displaystyle \displaystyle |\Pi_{\mathbb{R}}(S^{2}_{\sqrt{\frac{2}{6}}}\times S^{4}_{\sqrt{\frac{4}{6}}})|=\frac{64\pi^{3}}{81}$
\end{center}

\begin{center}
$\displaystyle \displaystyle |\Pi_{\mathbb{R}}(S^{3}_{\sqrt{\frac{3}{6}}}\times S^{3}_{\sqrt{\frac{3}{6}}})|=\frac{\pi^{4}}{4}$
\end{center}
thus

\begin{center}
 $\displaystyle \displaystyle W(\mathbb{R}P^{7})=\min \{\frac{25\sqrt{5}\pi^{4}}{216}, \frac{64\pi^{3}}{81},\frac{\pi^{4}}{4}\}=\frac{\pi^{4}}{4}$.
\end{center}
\end{itemize}
\begin{remark} Viana also computed the width of the real projective space \cite{celso} and Batista and Lima did it in dimensions less than or equal to $7$ \cite{batista, batista2}. 
\end{remark}

\begin{remark}
We do not apply Theorem \ref{xin} to $\mathbb{H}P^{r-1}$ because the only possible case is $\mathbb{H}P^{1}$ which is isometric to $S^{4}$ and its width is already computed.
\end{remark}

\end{proof}

\bibliographystyle{plain}
\bibliography{mybib}

\begin{thebibliography}{10}

\bibitem{almgren1965theory}
FJ~Almgren.
\newblock The theory of varifolds. mimeographed notes.
\newblock {\em Princeton University}, 1965.

\bibitem{ambrozio2018comparing}
Lucas Ambrozio, Alessandro Carlotto, and Ben Sharp.
\newblock Comparing the morse index and the first betti number of minimal
  hypersurfaces.
\newblock {\em Journal of Differential Geometry}, 108(3):379--410, 2018.

\bibitem{barbosa2012stability}
J~Lucas Barbosa, Manfredo Do~Carmo, and Jost Eschenburg.
\newblock Stability of hypersurfaces of constant mean curvature in riemannian
  manifolds.
\newblock In {\em Manfredo P. do Carmo--Selected Papers}, pages 291--306.
  Springer, 2012.

\bibitem{batista}
M\'arcio Batista and Anderson Lima.
\newblock Low min-max widths of the 3-dimensional real projective space.
\newblock {\em Personal communication}.

\bibitem{batista2}
M\'arcio Batista and Anderson Lima.
\newblock The widths of the real projective space.
\newblock {\em Personal communication}.

\bibitem{chodosh2016topology}
Otis Chodosh and Davi Maximo.
\newblock On the topology and index of minimal surfaces.
\newblock {\em Journal of Differential Geometry}, 104(3):399--418, 2016.

\bibitem{do2000compact}
Manfredo Do~Carmo, Manuel Ritor{\'e}, and Antonio Ros.
\newblock Compact minimal hypersurfaces with index one in the real projective
  space.
\newblock {\em Commentarii Mathematici Helvetici}, 75(2):247--254, 2000.

\bibitem{ketover2016catenoid}
Daniel Ketover, Fernando~C Marques, and Andr{\'e} Neves.
\newblock The catenoid estimate and its geometric applications.
\newblock {\em arXiv preprint arXiv:1601.04514}, 2016.

\bibitem{marques2012rigidity}
Fernando~C Marques and Andr{\'e} Neves.
\newblock Rigidity of min-max minimal spheres in three-manifolds.
\newblock {\em Duke Mathematical Journal}, 161(14):2725--2752, 2012.

\bibitem{marques2014min}
Fernando~C Marques and Andr{\'e} Neves.
\newblock Min-max theory and the willmore conjecture.
\newblock {\em Annals of mathematics}, pages 683--782, 2014.

\bibitem{morgan2002isoperimetric}
Frank Morgan and Manuel Ritor{\'e}.
\newblock Isoperimetric regions in cones.
\newblock {\em Transactions of the American Mathematical Society},
  354(6):2327--2339, 2002.

\bibitem{ohnita1986stable}
Yoshihiro Ohnita.
\newblock Stable minimal submanifolds in compact rank one symmetric spaces.
\newblock {\em Tohoku Mathematical Journal, Second Series}, 38(2):199--217,
  1986.

\bibitem{o1966fundamental}
Barrett O'Neill.
\newblock The fundamental equations of a submersion.
\newblock {\em The Michigan Mathematical Journal}, 13(4):459--469, 1966.

\bibitem{perdomo2001low}
Oscar Perdomo.
\newblock Low index minimal hypersurfaces of spheres.
\newblock {\em Asian Journal of Mathematics}, 5(4):741--750, 2001.

\bibitem{perdomo2004average}
Oscar Perdomo et~al.
\newblock On the average of the scalar curvature of minimal hypersurfaces of
  spheres with low stability index.
\newblock {\em Illinois Journal of Mathematics}, 48(2):559--565, 2004.

\bibitem{2019arXiv190210801P}
Oscar~M. {Perdomo}.
\newblock {On the index of minimal hypersurfaces of spheres}.
\newblock {\em arXiv e-prints}, page arXiv:1902.10801, Feb 2019.

\bibitem{pitts2014existence}
Jon~T Pitts.
\newblock {\em Existence and Regularity of Minimal Surfaces on Riemannian
  Manifolds.(MN-27)}.
\newblock Princeton University Press, 2014.

\bibitem{maleja1}
Alejandra Ramirez-Luna.
\newblock Orientability of min-max hypersurfaces in manifolds of positive
  {R}icci curvature.
\newblock {\em Personal communication}.

\bibitem{savo2010index}
Alessandro Savo.
\newblock Index bounds for minimal hypersurfaces of the sphere.
\newblock {\em Indiana University mathematics journal}, pages 823--837, 2010.

\bibitem{schoen1981regularity}
Richard Schoen and Leon Simon.
\newblock Regularity of stable minimal hypersurfaces.
\newblock {\em Communications on Pure and Applied Mathematics}, 34(6):741--797,
  1981.

\bibitem{simons1968minimal}
James Simons.
\newblock Minimal varieties in riemannian manifolds.
\newblock {\em Annals of Mathematics}, pages 62--105, 1968.

\bibitem{torralbo2014stable}
Francisco Torralbo and Francisco Urbano.
\newblock On stable compact minimal submanifolds.
\newblock {\em Proceedings of the American Mathematical Society},
  142(2):651--658, 2014.

\bibitem{urbano1990minimal}
Francisco Urbano.
\newblock Minimal surfaces with low index in the three-dimensional sphere.
\newblock {\em Proceedings of the American Mathematical Society}, pages
  989--992, 1990.

\bibitem{celso}
Celso {Viana}.
\newblock {Isoperimetry and volume preserving stability in real projective
  spaces}.
\newblock {\em arXiv e-prints}, page arXiv:1907.09445, Jul 2019.

\bibitem{zhou2015min}
Xin Zhou.
\newblock Min-max minimal hypersurface in $({M}^{n+1},g)$ with ${R}ic>0$ and
  $2\leq n \leq 6$.
\newblock {\em Journal of Differential Geometry}, 100(1):129--160, 2015.

\bibitem{zhou2017min}
Xin Zhou.
\newblock Min--max hypersurface in manifold of positive ricci curvature.
\newblock {\em Journal of Differential Geometry}, 105(2):291--343, 2017.

\end{thebibliography}

\end{document}